\documentclass[12pt]{article}

\usepackage{amsmath, amsthm, amssymb, mathrsfs, bm}
\usepackage{graphicx, subfig}
\usepackage{hyperref} 
\usepackage{booktabs}

\newtheorem{theorem}{Theorem}[section]
\newtheorem{definition}{Definition}[section]
\newtheorem{lemma}{Lemma}[section]
\newtheorem{problem}{Problem}[section]

\newtheorem{remark}{Remark}[section]

\graphicspath{{Figure/}}

\begin{document}

\title{Retrieval of acoustic sources from multi-frequency phaseless data}

\author{
Deyue Zhang\thanks{School of Mathematics, Jilin University, Changchun,
P. R. China. {\it dyzhang@jlu.edu.cn}}, 
Yukun Guo\thanks{Department of Mathematics, Harbin Institute of Technology, Harbin, P. R. China. {\it ykguo@hit.edu.cn} (Corresponding author)}, 
Jingzhi  Li\thanks{Department of Mathematics, Southern University of Science and Technology, Shenzhen, P. R. China. {\it li.jz@sustc.edu.cn}}\ \ and
Hongyu Liu\thanks{Department of Mathematics, Hong Kong Baptist University, Hong Kong SAR, P. R. China. {\it hongyuliu@hkbu.edu.hk}}
}

\maketitle

\begin{abstract}
This paper is concerned with the inverse source problem of reconstructing an unknown acoustic excitation from phaseless measurements of the radiated fields away at multiple frequencies. It is well known that the non-uniqueness issue is a major challenge associated with such an inverse problem. We develop a novel strategy to overcome this challenging problem by recovering the radiated fields via adding some reference point sources as extra artificial sources to the inverse source system. This novel reference source technique requires only a few extra data, and brings in a simple phase retrieval formula. The stability of this phase retrieval approach is rigorously analyzed. After the reacquisition of the phase information, the multi-frequency inverse source problem with recovered phase information is solved by the Fourier method, which is non-iterative, fast and easy to implement. Several numerical examples are presented to demonstrate the feasibility and effectiveness of the proposed method.
\end{abstract}



\section{Introduction}

The inverse source problems concerning the reconstruction of an unknown source from measurements of the radiated field arise naturally in a broad range of scientific and engineering applications, such as antenna synthesis \cite{AKK91, DML07}, acoustic tomography \cite{AZML07, LU15} and medical imaging \cite{AM06, ABF02, Arr99, DLU17, HR98}. Recently, much attention \cite{AKK91, BC77, DML07, El1, El2, EV09, ZGLL17} has been focused on the inverse source problem of determining an acoustic source. Motivated by uniqueness and stability results, several numerical methods for solving the multi-frequency inverse source problem have been proposed. We refer interested readers to \cite{AHLS17, BLLT15, BLRX15, EV09, WGZL17, ZG15} for the sampling method, the continuation method, the eigenfunction expansion method and the Fourier method, etc.

Most existing reconstruction algorithms require the measurements of full data (both the intensity and phase information). However, in many practical applications, the measurements of the full data cannot be obtained, but instead only the intensity or modulus information of the data is available. This is particularly the case in the high-frequency regime. Hence, there are growing efforts devoted to the recovery with phaseless data. Recently,  in \cite{BLT11}, the authors presented the continuation method for an inverse source problem with the phaseless measurements. Several numerical methods for the inverse scattering from phaseless data can be found in \cite{BZ16, IK11, KR17, ZZ17}.

This paper is concerned with the two dimensional inverse problem of reconstructing an acoustic source from multi-frequency phaseless data. More precisely, we describe the mathematical model of the inverse problem in what follows.

Let $S$ be a source to the homogeneous Helmholtz equation
\begin{equation}\label{HelmholtzS}
 \Delta u +k^2u=S \quad \mathrm{in} \ \mathbb{R}^2,
\end{equation}
where $k\in \mathbb{R}_+$ is the wavenumber. Assume that $S$ is independent of $k$ and
\begin{equation}\label{SourceS}
S\in L^2(\mathbb{R}^2),\quad \mathrm{supp} S \subset\subset V_0,
\end{equation}
and the radiated field $u$ satisfies the Sommerfeld radiation condition
\begin{equation}\label{Sommerfeld}
\frac{\partial u}{\partial r}-\mathrm{i} ku=o\left(r^{-1/2}\right), \quad r=|x| \to \infty,
\end{equation}
where $V_0$ is a rectangle centered at the origin. Then, there is a unique solution $u$ satisfying the Helmholtz equation \eqref{HelmholtzS} and the radiation condition \eqref{Sommerfeld}.
In addition,
\begin{equation}\label{SolutionU}
u(x)=-\int_{V_0}\Phi_k(x,y)S(y)\,\mathrm{d} y, \quad \Phi_k(x,y):=\frac{\mathrm{i}}{4}H^{(1)}_0(k|x-y|),
\end{equation}
where $H^{(1)}_0$ is the Hankel function of the first kind of order zero. In the following, we also employ $u(x, k)$ to indicate the dependence of the radiated field $u(x)$ on the wavenumber $k$. Assume that the radiated fields are measured on the acquisition curve $\Gamma_R:=\{x \in \mathbb{R}^2: |x|=R\}$ and $V_0\subset\subset B_R:=\{x \in \mathbb{R}^2: |x|<R\}$. Then the inverse problem considered in this paper is stated as follows:

\begin{problem}[Phaseless multifrequency inverse source problem]\label{problem1}
Let $N\in \mathbb{N}_+$ and $\mathbb{K}_N$ be an admissible set consisting of a finite number of wavenumbers. Then the inverse source problem is to construct an approximation for the source $S(x)$ from the multifrequency phaseless data $\{|u(x, k)|: x\in\Gamma_R ,\ k\in\mathbb{K}_N\}$.
\end{problem}

There is a major difficulty in solving this problem: non-uniqueness, i.e., the approximation cannot be uniquely determined. This is due to the fact that the sources $S(x)$ and  $-S(x)$ produce the same phaseless data
\[
\left|\int_{V_0}\Phi_k(x,y)S(y)\,\mathrm{d} y\right|
\]
for all $k\in \mathbb{R}_+$ and  $x\in \mathbb{R}^2$.

This paper is devoted to overcoming this difficulty by decomposing  {\bf Problem} \ref{problem1} into two subproblems: first, recovering the phase information of the radiated fields and next reconstructing the source term. Motivated by the reference ball technique in inverse scattering problems \cite{LLZ09}, we develop a reference source technique to retrieve the phase information of the measured data.  To this end, we incorporate some artificially added point sources, respectively, as additional reference sources into the inverse source system, and derive a system of equations for the radiated fields, which can be solved directly. Moreover, we choose the point sources with suitable locations such that the absolute value of the determinant of the coefficient matrix admits a positive lower bound, which ensures the stability of the system. Then we can reconstruct the source from the recovered radiated fields by using the Fourier method developed in \cite{ZG15}.

The outline of this paper is as follows.  In the next section, we present the reference source technique and establish a phase retrieval formula for the radiated fields.  Section \ref{sec:stability} is devoted to analyzing the stability of the method. Section \ref{sec:fourier_method} serves as a brief review of the Fourier method. In Section \ref{sec:numerics}, we present several numerical examples to show the effectiveness and efficiency of the proposed method. Finally, some concluding remarks are included in Section \ref{sec:conclusion}.

\section{Reference source technique and phase retrieval formula}\label{sec:phase_retrieval}

We begin this section with some notation and definitions that are used in the paper. Without loss of generality, let
\[
V_0=(-a, a)^2, \quad  a>0.
\]
Let $m$ be a positive integer and $\vartheta_\mu=\mu\pi/m$ for $\mu\in \mathbb{R}$. Denote
\begin{align*}
B_j:=&\{r(\cos \theta, \sin\theta): 0\leq r\leq R,\ \vartheta_{2j-2} \leq\theta\leq \vartheta_{2j} \},\quad j=1,\cdots,m, \\
\Gamma_j:= & \{R(\cos \theta, \sin\theta): \vartheta_{2j-2} \leq\theta\leq \vartheta_{2j}\},\quad j=1,\cdots,m.
\end{align*}
Then, we have
\[
B_R=\bigcup\limits_{j=1}^m B_j\quad \mathrm{and}\quad \Gamma_R=\bigcup\limits_{j=1}^m\Gamma_j.
\]
The geometrical configuration is illustrated in Figure \ref{fig:Setting}.
\begin{figure}
\centering
\includegraphics[width=0.9\linewidth]{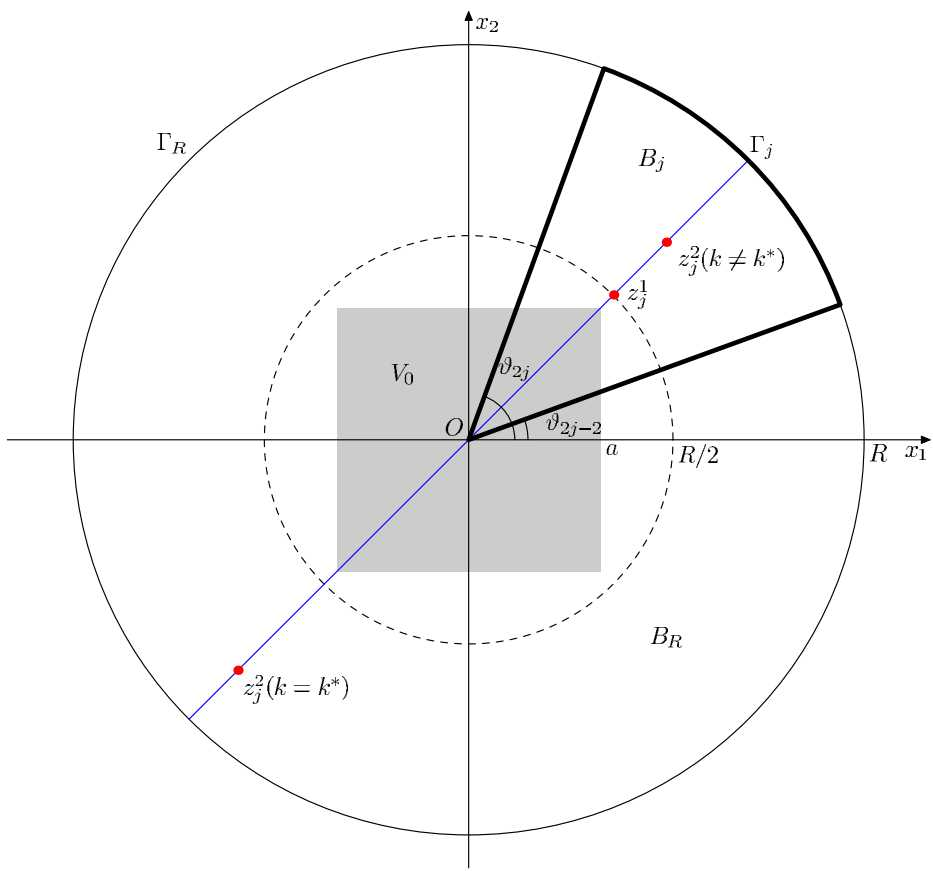}
\caption{An illustration of the geometrical setting for the reference source technique. The locations of $z_{j,\ell}, \ell=1,2$, are marked by small red points.}
\label{fig:Setting}
\end{figure}

For each $j$ ($j=1,\cdots,m$), take two points $z_{j,\ell}:=\lambda_{j,\ell} R(\cos\vartheta_{2j-1}, \sin\vartheta_{2j-1})$, where $\lambda_{j,\ell}\in\mathbb{R}, \sqrt{2}a/R\leq |\lambda_{j,\ell}|<1$ and $\ell=1,2$. Let $\delta_{j,\ell}$ ($j=1,\cdots, m,  \ell=1,2$)
be the Dirac distributions at points $z_{j,\ell}$, and introduce the scaling factors
\begin{equation}\label{ratio}
  c_{j,\ell, k}:=\frac{\|u(\cdot, k)\|_{j,\infty}}{\|\Phi_k(\cdot, z_{j,\ell})\|_{j,\infty}}, \quad
   j=1,\cdots, m, \quad \ell=1,2,
\end{equation}
where $\|\cdot\|_{j,\infty}:= \|\cdot\|_{L^\infty(\Gamma_j)}$.
Then $\Psi_{j,\ell}(x,k):=-c_{j,\ell,k}\Phi_k(x,z_{j,\ell})$ satisfies the following inhomogeneous Helmholtz equation
\[
 \Delta \Psi_{j,\ell} +k^2\Psi_{j,\ell}=c_{j,\ell,k}\delta_{j,\ell} \quad \mathrm{in} \ \mathbb{R}^2.
\]
Further, $v_{j,\ell}:=u+\Psi_{j,\ell}$ is the unique solution to the problem
\begin{equation*}
\begin{cases}
   \Delta v_{j,\ell} +k^2v_{j,\ell}=S+\delta_{j,\ell}, &\mathrm{in} \ \mathbb{R}^2,\\
   \dfrac{\partial v_{j,\ell}}{\partial r}-\mathrm{i} kv_{j,\ell}=o\left(r^{-1/2}\right),
 & r=|x| \rightarrow \infty.
\end{cases}
\end{equation*}

Since the phase retrieval technique is designated for the Fourier method \cite{ZG15} in this paper, we recall the following definition of admissible frequencies.
\begin{definition}[\cite{ZG15}]
Let $N\in \mathbb{N}_+$ and $k^*\in\mathbb{R}_+$ be a small wavenumber such that $0<k^*R<1$. Then the admissible set of wavenumbers is given by
\begin{equation*}
\mathbb{K}_N:=\left\{\frac{\pi}{a}|\bm{\ell}| : \bm{\ell}\in \mathbb{Z}^2, 1\leq |\bm{\ell}|_{\infty}\leq N\right\}\bigcup \{k^*\}.
\end{equation*}
\end{definition}

With these preparations, we introduce the phase retrieval problem in this paper:

\begin{problem}[Phase retrieval]
Given $N\in\mathbb{N}_+$ and the phaseless data
\begin{align*}
&\{|u(x,k)| : x\in\Gamma_R,\ k\in\mathbb{K}_N\},\\
& \{|v_{j,\ell}(x,k)| : x\in\Gamma_j,\ k\in\mathbb{K}_N\},\quad j=1,\cdots,m,\quad\ell=1,2,
\end{align*}
recover the radiated fields $\{u(x,k) : x\in\Gamma_R,\ k\in\mathbb{K}_N\}$.
\end{problem}

Now, we turn to the description of our phase retrieval method. For simplicity, we denote
\[
u_{j,k}(\cdot):=u(\cdot,k)|_{\Gamma_j},\quad v_{j,\ell,k}(\cdot):=v_{j,\ell}(\cdot,k)|_{\Gamma_j},\quad \ell=1,2,
\]
for each fixed $j=1,\cdots,m$ and $k\in\mathbb{K}_N$.

Since $H^{(1)}_{0}(z)=J_0(z)+\mathrm{i} Y_0(z)$ for $z\in \mathbb{C}\setminus\{0\}$, we have
\begin{eqnarray*}
{\rm Re}(\Psi_{j,\ell}(x, k))=\frac{c_{j,\ell, k}}{4}Y_0(kr_{j,\ell}),\quad
{\rm Im}(\Psi_{j,\ell}(x, k))=-\frac{c_{j,\ell, k}}{4}J_0(kr_{j,\ell}),
\end{eqnarray*}
where $r_{j,\ell}=|x-z_{j,\ell}|$, $J_{0}$ and $Y_0$ are the Bessel function of the
first and second kind of order 0, respectively. Then, by using the phaseless measurements
$|u_{j, k}|, |v_{j,1,k}|$ and $|v_{j,2,k}|$, we obtain
\begin{align}
&({\rm Re}(u_{j, k}))^2+({\rm Im}(u_{j, k}))^2=|u_{j, k}|^2, \label{EqnO1} \\
&\left({\rm Re}(u_{j, k})+\frac{c_{j,1,k}}{4}Y_0(kr_{j,1})\right)^2+\left({\rm Im}(u_{j, k})-\frac{c_{j,1,k}}{4}J_0(kr_{j,1})\right)^2=|v_{j,1,k}|^2, \label{EqnO2} \\
&\left({\rm Re}(u_{j, k})+\frac{c_{j,2,k}}{4}Y_0(kr_{j,2})\right)^2+\left({\rm Im}(u_{j, k})-\frac{c_{j,2,k}}{4}J_0(kr_{j,2})\right)^2=|v_{j,2,k}|^2. \label{EqnO3}
\end{align}
Further, by subtracting \eqref{EqnO1} from \eqref{EqnO2} and \eqref{EqnO3}, respectively, we have
\begin{align}
Y_0(kr_{j,1}){\rm Re}(u_{j,k}) - J_0(kr_{j,1}){\rm Im}(u_{j,k})  =  f_{j,1,k}, \label{EqnF1}\\
Y_0(kr_{j,2}){\rm Re}(u_{j,k}) - J_0(kr_{j,2}){\rm Im}(u_{j,k})  =  f_{j,2,k}, \label{EqnF2}
\end{align}
where
\[
 f_{j,\ell, k}=\frac{2}{c_{j,\ell,k}}\left(|v_{j,\ell,k}|^2-|u_{j,k}|^2\right)-\frac{c_{j,\ell,k}}{8}\left|H^{(1)}_{0}(kr_{j,\ell})\right|^2,
\quad \ell=1,2.
\]
Thus, by a simple calculation, we derive the  phase retrieval formula on $\Gamma_j$
\begin{equation}\label{formula}
  {\rm Re}(u_{j,k}) =\frac{\det(A_{j,k}^{\rm R})}{\det(A_{j,k})}, \quad
  {\rm Im}(u_{j,k}) = \frac{\det(A_{j,k}^{\rm I})}{\det(A_{j,k})},
\end{equation}
where the function matrices $A_{j, k}$, $A_{j,k}^{\rm R}$ and $A_{j,k}^{\rm I}$ are defined as follows
\begin{align*}
A_{j,k}=&
\begin{pmatrix}
    Y_0(kr_{j,1}) & -J_0(kr_{j,1}) \\
    Y_0(kr_{j,2}) & -J_0(kr_{j,2}) \\
\end{pmatrix}, \quad
A_{j,k}^{\rm R}=
\begin{pmatrix}
     f_{j,1,k} & -J_0(kr_{j,1}) \\
     f_{j,2,k} & -J_0(kr_{j,2}) \\
\end{pmatrix} \\
A_{j,k}^{\rm I}=&
\begin{pmatrix}
    Y_0(kr_{j,1}) &  f_{j,1,k}  \\
    Y_0(kr_{j,2}) &  f_{j,2,k}  \\
\end{pmatrix}
\end{align*}

Therefore, the radiated field $u_{j,k}$ can be recovered from
$u_{j,k}={\rm Re}(u_{j, k})+\mathrm{i} {\rm Im}(u_{j, k})$ for $j=1,\cdots,m$.

In this paper, the parameters $R, m$ and $\lambda_{j,\ell}$ for $\ell=1,2$ are taken as follows:
\begin{equation}\label{assumption}
\begin{array}{lr}
 m\geq10, \quad  \lambda_{j,1}=\dfrac{1}{2}, \quad k^*=\dfrac{\pi}{30a},\quad  \tau \geq 6,
\\
 \begin{cases}
    R=\tau a,  \quad\lambda_{j,2}=\dfrac{1}{2}+\dfrac{\pi}{2kR},  &  \mathrm{if}\ k\in \mathbb{K}_N\backslash\{k^*\}, \\
    R=6 a,  \quad\lambda_{j,2}= -\dfrac{3}{2},  &  \mathrm{if}\ k=k^*.
 \end{cases}
\end{array}
\end{equation}

The choice of the parameters \eqref{assumption} will make formula \eqref{formula} meaningful or the equations \eqref{EqnF1} and \eqref{EqnF2} uniquely solvable,
which will be discussed in the next section.

However, formulas \eqref{formula} cannot be used directly in computation since, in
general, only the noisy data can be measured. Therefore, at the end of this section,
we describe the algorithm with perturbed data.

\begin{table}[h]
\centering
\begin{tabular}{cp{.8\textwidth}}
\toprule
\multicolumn{2}{l}{{\bf Algorithm PR:}\quad Phase retrieval with reference point sources} \\
\midrule
 {\bf Step 1} & Take the parameters $R, m$ and $\lambda_{j,\ell}$ for $j=1,\cdots,m, \ell=1,2$ as in  \eqref{assumption};  \\
{\bf Step 2} &  Measure the noisy phaseless data $\{|u^\epsilon(x,k)|: x\in\Gamma_R,\ k\in\mathbb{K}_N\}$ and compute the scaling factors $c_{j,\ell,k}^\epsilon$ for $j=1,\cdots,m, \ell=1,2$ and $k\in\mathbb{K}_N$; \\
{\bf Step 3} & Introduce the reference point sources $c_{j,\ell,k}^\epsilon\delta_{j,\ell},\  j=1,\cdots,m, \ell=1,2$, respectively, to the inverse source system $S$, and collect the phaseless data $\{|v_{j,\ell}^\epsilon(x,k)|: x\in\Gamma_j,\ k\in\mathbb{K}_N\}$ for $j=1,\cdots,m, \ell=1,2$;\\
{\bf Step 4} & Recover the radiated field $\{u^\epsilon(x,k): x\in\Gamma_j,\ k\in\mathbb{K}_N\}$ for $j=1,\cdots,m$ from formula \eqref{formula}. \\
\bottomrule
\end{tabular}
\end{table}

\section{Stability of phase retrieval}\label{sec:stability}

In this section, we analyze the stability of the phase retrieval method. We start with the following asymptotic expansion of the Hankel functions $H^{(1)}_0(z)$.

\begin{lemma}\label{Expansion}
Let $t>0$, then the following estimate holds
\begin{equation}\label{Estimate}
  \left|H^{(1)}_0(t)-\sqrt{\frac{2}{\pi t}}e^{\mathrm{i} (t-\pi/4)}\right| \leq  \frac{t^{-3/2}}{4\sqrt{2\pi}}.
\end{equation}
\end{lemma}

\begin{proof}
From the integral representations of $H^{(1)}_{\nu}(z)$ (see \cite[Section 7.2]{Watson} and \cite[Chapter 7, Section 13.3]{Olv74})
and the assumption $t>0$, we have
\begin{equation}\label{Hankel1}
  H^{(1)}_{0}(t)=\sqrt{\frac{2}{\pi t}}\frac{\mathrm{e}^{\mathrm{i} (t-\pi/4)}}{\Gamma(\frac{1}{2})}\int_0^\infty \mathrm{e}^{-s}s^{-1/2}\left(1+\frac{\mathrm{i} s}{2t}\right)^{-1/2}\, \mathrm{d} s
\end{equation}
where $\Gamma(\cdot)$ denotes the Gamma function
\[
\Gamma(x):=\int_0^\infty\mathrm{e}^{-s}s^{x-1}\,\mathrm{d}s.
\]
Since
\begin{equation}\label{kernel1}
  \left(1+\frac{\mathrm{i} s}{2t}\right)^{-1/2}
=1-\frac{\mathrm{i} s}{4t}\int_0^1\left(1+\frac{\mathrm{i} s\tau}{2t}\right)^{-3/2}\,\mathrm{d} \tau.
\end{equation}
and $t>0$, $s\geq0$, $\tau\geq0$ in \eqref{kernel1}, it is readily to see
\begin{equation}\label{kernel2}
  \left|1+\frac{\mathrm{i} s\tau}{2t}\right|^{-3/2} \leq 1.
\end{equation}
Substituting \eqref{kernel1} into \eqref{Hankel1} and using \eqref{kernel2}, we deduce
\begin{align*}
 \!\!\! \left|H^{(1)}_{0}(t)-\sqrt{\frac{2}{\pi t}}e^{\mathrm{i} (t-\pi/4)}\right|
   \leq & \sqrt{\frac{2}{\pi t}}\frac{1}{\Gamma(\frac{1}{2})}
\int_0^\infty \mathrm{e}^{-s}s^{-1/2}\frac{s}{4t}\int_0^1\left|1+\frac{\mathrm{i} s\tau}{2t}\right|^{-3/2}\mathrm{d} \tau\mathrm{d} s \\
 \leq & \sqrt{\frac{2}{\pi t}}\frac{1}{4t\Gamma(\frac{1}{2})}
\int_0^\infty \mathrm{e}^{-s}s^{1/2}\mathrm{d} s \\
= & \sqrt{\frac{2}{\pi t}}\frac{\Gamma(\frac{3}{2})}{4t\Gamma(\frac{1}{2})} \\
= & \frac{t^{-3/2}}{4\sqrt{2\pi}},
\end{align*}
which leads to the estimate \eqref{Estimate}.  The proof is complete. 
\end{proof}

With Lemma \ref{Expansion}, we are in a position to derive an estimate on $\det(A_{j,k})$,  which shall be of critical importance in our subsequent stability analysis.

\begin{theorem}\label{det}
Under \eqref{assumption},  we have the following estimate
\begin{equation}\label{detAj}
  |\det(A_{j,k})| \geq 
\begin{cases}
\dfrac{M}{kR}, & \mathrm{if}\ k\in\mathbb{K}_N\backslash\{k^*\}, \\
 M^*,  & \mathrm{if}\ k=k^*,
\end{cases}
\quad j=1,\cdots, m,
\end{equation}
where $M= 1-\frac{7}{20\tau}$ and $M^*=\frac{4}{9}$.
\end{theorem}

\begin{proof} We divide the proof into two cases.

Case (i): $k\in\mathbb{K}_N\backslash\{k^*\}$. 

It is clear that
\begin{eqnarray*}
k\geq k_{\min}:=\min\limits_{k\in \mathbb{K}_N\backslash\{k^*\}}k=\frac{\pi}{a}.
\end{eqnarray*}
And from \eqref{assumption}, we see
\begin{equation*}
  \lambda_{j,2} \leq \frac{1}{2}+\frac{\pi}{2k_{\min}\tau a}=\frac{\tau+1}{2\tau}<\frac{2}{3}.
\end{equation*}
By using the definition $r_{j,\ell}=|x-z_{j,\ell}|$ for $\ell=1,2$ and $x=R(\cos\theta,\sin\theta)\in \Gamma_j$, we have
\begin{eqnarray}\label{rjl}
      r_{j,\ell}=R\sqrt{1+\lambda_{j,\ell}^2-2\lambda_{j,\ell}\cos(\theta-\vartheta_{2j-1})}
          \geq (1-\lambda_{j,\ell})R,
\end{eqnarray}
which, together with \eqref{assumption}, yields
\begin{align}
  k r_{j,1} \geq& kR(1-\lambda_{j,1})\geq k_{\min} R \left(1-\frac{1}{2}\right)
      = \frac{\pi}{2}\frac{R}{a}=\frac{\tau\pi}{2}, \label{krj1} \\
  k r_{j,2} \geq& kR(1-\lambda_{j,2})> k_{\min} R \left(1-\frac{2}{3}\right) \label{krj2}
    =\frac{\tau\pi}{3},
\end{align}
and
\begin{align*}
k(r_{j,1}-r_{j,2})=&k\frac{r_{j,1}^2-r_{j,2}^2}{r_{j,1}+r_{j,2}} \\
=& kR^2(\lambda_{j,2}-\lambda_{j,1})\frac{2\cos(\theta-\vartheta_{2j-1})-(\lambda_{j,1}+\lambda_{j,2})}{r_{j,1}+r_{j,2}} \\
=& \frac{\pi}{2}R \frac{2\cos(\theta-\vartheta_{2j-1})-(\lambda_{j,1}+\lambda_{j,2})}{r_{j,1}+r_{j,2}} \\
=& \frac{\pi}{2}\frac{2\cos(\theta-\vartheta_{2j-1})-\lambda_{j,2}-\frac{1}{2}}{\sqrt{\frac{5}{4}-\cos(\theta-\vartheta_{2j-1})}+\sqrt{1+\lambda_{j,2}^2-2\lambda_{j,2}\cos(\theta-\vartheta_{2j-1})}} \\
=:& \frac{\pi}{2} \varphi(\theta,\lambda_{j,2}).
\end{align*}

Further, from $\vartheta_{2j-2} \leq\theta\leq \vartheta_{2j}$, $m\geq10$ and a direct calculation, we have
\begin{equation*}
  \frac{\partial\varphi}{\partial\theta}
\begin{cases}
   \geq 0, & \mathrm{if} \ \vartheta_{2j-2} \leq\theta\leq \vartheta_{2j-1}, \\
   \leq 0, & \mathrm{if} \ \vartheta_{2j-1} \leq\theta\leq \vartheta_{2j}.
\end{cases}
\end{equation*}

This means $\varphi(\vartheta_{2j},\lambda_{j,2}) \leq \varphi(\theta,\lambda_{j,2}) \leq \varphi(\vartheta_{2j-1},\lambda_{j,2})=1,\forall \lambda_{j,2}\in (\frac{1}{2},\frac{2}{3}]$. In addition,
\begin{align*}
\varphi(\vartheta_{2j},\lambda_{j,2})\geq& \frac{2\cos(\frac{\pi}{10})-\lambda_{j,2}-\frac{1}{2}}{\sqrt{\frac{5}{4}-\cos(\frac{\pi}{10})}+\sqrt{1+\lambda_{j,2}^2-2\lambda_{j,2}\cos(\frac{\pi}{10})}}\\
\geq& \frac{2\cos(\frac{\pi}{10})-\frac{2}{3}-\frac{1}{2}}{\sqrt{\frac{5}{4}-\cos(\frac{\pi}{10})}+\sqrt{1+(\frac{1}{2})^2-2\cdot\frac{1}{2}\cos(\frac{\pi}{10})}}\\
\approx& 0.67255
\end{align*}
Hence,
\begin{equation}\label{kr1r2}
  0.3363\pi \leq  k (r_{j,1}-r_{j,2})\leq \frac{\pi}{2}.
\end{equation}

Denote
\begin{equation}\label{alpha}
  \alpha(t):=J_{0}(t)-\sqrt{\frac{2}{\pi t}}\cos\left(t-\frac{\pi}{4}\right),
 \quad \beta(t):=Y_{0}(t)-\sqrt{\frac{2}{\pi t}}\sin\left(t-\frac{\pi}{4}\right).
\end{equation}
Then by using \eqref{Estimate}, we obtain
\begin{equation}\label{detEstimate1}
  \sqrt{|\alpha(t)|^2+ |\beta(t)|^2} \leq \frac{t^{-3/2}}{4\sqrt{2\pi}},  \qquad t>0.
\end{equation}
From  \eqref{formula} and \eqref{alpha}, it is readily to see that
\begin{align*}
  \det(A_{j,k})=& J_0(kr_{j,1})Y_0(kr_{j,2}) - J_0(kr_{j,2}) Y_0(kr_{j,1}) \\
    =& \frac{2}{\pi k\sqrt{r_{j,1}r_{j,2}}}\sin\left(k(r_{j,2}-r_{j,1})\right)+\gamma(s_1,s_2),
\end{align*}
where $s_\ell=k r_{j,\ell}$ ($\ell=1,2$), and
\begin{align*}
\gamma(s_1,s_2) =& \sqrt{\frac{2}{\pi s_2}} \left( \alpha(s_1)\sin\left(s_2-\frac{\pi}{4}\right)-\beta(s_1)\cos\left(s_2-\frac{\pi}{4}\right) \right) \\
    &- \sqrt{\frac{2}{\pi s_1}} \left( \alpha(s_2)\sin\left(s_1-\frac{\pi}{4}\right)
                        -\beta(s_2)  \cos\left(s_1-\frac{\pi}{4}\right)\right) \\
    &+ \alpha(s_1)\beta(s_2)+\alpha(s_2)\beta(s_1).
\end{align*}
Further by using \eqref{detEstimate1}, \eqref{krj1} and \eqref{krj2}, we derive
\begin{align*}
  |\gamma(s_1,s_2)| \leq& \sqrt{\frac{2}{\pi s_2}} \left( |\alpha(s_1)|
                       +|\beta(s_1)|  \right) +  \left|\alpha(s_1)\beta(s_2)\right| \\
  &+ \sqrt{\frac{2}{\pi s_1}} \left( |\alpha(s_2)| + |\beta(s_2)| \right)
   +\left|\alpha(s_2)\beta(s_1)\right| \\
  \leq& \sqrt{\frac{2}{\pi s_2}}\frac{s_1^{-3/2}}{4\sqrt{\pi}}
         +\sqrt{\frac{2}{\pi s_1}}\frac{s_2^{-3/2}}{4\sqrt{\pi}}
   +2\frac{s_1^{-3/2}}{4\sqrt{2\pi}}\frac{s_2^{-3/2}}{4\sqrt{2\pi}} \\
   =& \frac{1}{2\sqrt{2}\pi k\sqrt{r_{j,1}r_{j,2}}}\left(\frac{1}{s_1}+\frac{1}{s_2}
     + \frac{1}{4\sqrt{2}s_1s_2}\right) \\
   \leq& \frac{1}{2\sqrt{2}\pi k\sqrt{r_{j,1}r_{j,2}}}\left(\frac{2}{\tau\pi}+\frac{3}{\tau\pi}+\frac{3}{2\sqrt{2}\tau^2\pi^2}\right) \\
\leq& \frac{1}{2\sqrt{2}\pi \tau k\sqrt{r_{j,1}r_{j,2}}}\left(\frac{5}{\pi}+\frac{1}{2\sqrt{2}\pi^2}\right) \\
\approx& \frac{0.5754}{\tau\pi k\sqrt{r_{j,1}r_{j,2}}}\\
   <& \frac{3}{5\tau\pi k\sqrt{r_{j,1}r_{j,2}}},
\end{align*}
which, together with \eqref{kr1r2}, $r_{j,1}>r_{j,2}$ and $\lambda_{j,1} = 1/2$, implies
\begin{align*}
    |\det(A_{j,k})| \geq& \frac{2}{\pi k\sqrt{r_{j,1}r_{j,2}}}|\sin\left(k(r_{j,2}-r_{j,1})\right)|-|\gamma(s_1, s_2)| \\
  \geq& \frac{2}{\pi k\sqrt{r_{j,1}r_{j,2}}}
   \left( \left|\sin\left(k(r_{j,2}-r_{j,1})\right)\right|-  \frac{3}{10\tau}\right) \\
  \geq& \frac{2}{\pi k r_{j,1}}\left( \sin(0.3363\pi)-  \frac{3}{10\tau} \right) \\
  \geq& \frac{2\left( \sin(0.3363\pi)-  \frac{3}{10\tau} \right)}{\pi kR\sqrt{1+\lambda_{j,1}^2-2\lambda_{j,1}\cos(\frac{\pi}{10})}} \\
\approx& \frac{1}{kR}\left(1.0137-\frac{6.9861}{20\tau}\right)\\
 \geq& \frac{M}{kR},
\end{align*}
where $M= 1- \frac{7}{20\tau}$. Hence, the estimate \eqref{detAj} is obtained.

Case (ii): $k=k^*=\pi/(30a)$. 

Denote
\begin{align}
 J_0(t)=&1-\frac{t^2}{4}+\tilde{\alpha}(t), \label{alpha1} \\
 Y_0(t)=&\frac{2}{\pi}\left(1-\frac{t^2}{4}\right)\left(\ln\frac{t}{2}+C_0\right)+ \frac{t^2}{2\pi}+\tilde{\beta}(t), \label{alpha2}
\end{align}
where $C_0=0.5772\cdots$ is Euler's constant. Since
\[
 J_0(t)=\sum_{p=0}^{\infty}\frac{(-1)^p}{(p!)^2}\left(\frac{t}{2}\right)^{2p}
\]
and
\[
 Y_0(t)=\frac{2}{\pi} J_0(t)\left(\ln\frac{t}{2}+C_0\right)-\frac{2}{\pi}\sum_{p=1}^{\infty}
     (-1)^p\frac{\left(1+\frac{1}{2}+\cdots+\frac{1}{p}\right)}{(p!)^2}\left(\frac{t}{2}\right)^{2p},
\]
we have
\begin{equation}\label{detEstimate2}
  0<\tilde{\alpha}(t) \leq \frac{t^4}{64},
    \quad |\tilde{\beta}(t)|\leq  \frac{ t^3}{72}+\frac{ t^4}{62}, \quad 0<t<2.
\end{equation}
From \eqref{formula}, \eqref{alpha1} and \eqref{alpha2}, we see
\begin{align*}
  \det(A_{j,k})=& J_0(t_1)Y_0(t_2) - J_0(t_2) Y_0(t_1) \\
    =& \frac{2}{\pi} \left(1-\frac{t_1^2}{4}\right)\left(1-\frac{t_2^2}{4}\right)\ln \frac{r_{j,2}}{r_{j,1}}+\frac{t_2^2-t_1^2}{2\pi}+\tilde{\gamma}(t_1, t_2),
\end{align*}
where $t_\ell=k^*r_{j,\ell}$ ($\ell=1,2$), and
\begin{align*}
  \gamma(t_1, t_2) =&  \tilde{\alpha}(t_1)
       \left( \frac{2}{\pi}\left(1-\frac{t_2^2}{4}\right)\left(\ln\frac{t_2}{2}+ C_0\right)
                              +\frac{t_2^2}{2\pi}+\tilde{\beta}(t_2) \right) \\
                      & -\tilde{\alpha}(t_2)
       \left( \frac{2}{\pi}\left(1-\frac{t_1^2}{4}\right)\left(\ln\frac{t_1}{2}+ C_0\right)
                               +\frac{t_1^2}{2\pi} +\tilde{\beta}(t_1) \right) \\
        &+\left(1-\frac{t_1^2}{4}\right)\tilde{\beta}(t_2)-\left(1-\frac{t_2^2}{4}\right)\tilde{\beta}(t_1).
\end{align*}
Further,  using \eqref{detEstimate2}, $k^*R=\pi/5, 0.5k^*R \leq t_1\leq 0.55k^*R$ and
$2.47k^*R \leq t_2\leq 2.5k^*R$, we derive
\begin{align*}
    \left|\gamma(t_1, t_2)\right| \leq&  \frac{t_1^4}{64}
             \left(\frac{2}{\pi}\left(1-\frac{t_2^2}{4}\right)\left|\ln\frac{t_2}{2}+C_0\right|
                        +\frac{t_2^2}{2\pi} +\frac{ t_2^3}{72}+\frac{ t_2^4}{62}\right) \\
  & + \frac{t_2^4}{64}
             \left(\frac{2}{\pi}\left(1-\frac{t_1^2}{4}\right)\left|\ln\frac{t_1}{2}+C_0\right|
                        +\frac{t_1^2}{2\pi} +\frac{ t_1^3}{72}+\frac{ t_1^4}{62}\right) \\
  & + \left(1-\frac{t_1^2}{4}\right)\left(\frac{ t_2^3}{72}+\frac{ t_2^4}{62}\right)
      + \left(1-\frac{t_2^2}{4}\right)\left(\frac{ t_1^3}{72}+\frac{ t_1^4}{62}\right) \\
   \leq& 1.43\times10^{-4}+0.0767+0.1483+ 3.25\times10^{-4} \\
   <& 0.2255,
\end{align*}
and thus,
\begin{align*}
  \left|\det(A_{j,k})\right| \geq&  \frac{2}{\pi} \left(1-\frac{t_1^2}{4}\right)\left(1-\frac{t_2^2}{4}\right)\ln \frac{r_{j,2}}{r_{j,1}}
    + \frac{t_2^2-t_1^2}{2\pi}-\left|\tilde{\gamma}(k^*,r_{j,1},r_{j,2})\right| \\
 >&  0.3552+0.3632-0.2255 \\
 >& \frac{4}{9}.
\end{align*}
This completes the proof.  
\end{proof}

\begin{remark}
We would like to point out that there exist other strategies for choosing the parameters in \eqref{assumption} such that the inequality \eqref{detAj} holds, and the unique solvability of equations \eqref{EqnF1} and \eqref{EqnF2} could be ensured as well.
\end{remark}

Now, we turn to analyzing the stability of the phase retrieval formula.
For a fixed $k$ and $j$, we consider the perturbed equation system for the unknowns ${\rm Re}(u_{j,k}^\epsilon)$ and ${\rm Im}(u_{j,k}^\epsilon)$:
\begin{align}
Y_0(kr_{j,1}){\rm Re}(u_{j,k}^\epsilon) - J_0(kr_{j,1}){\rm Im}(u_{j,k}^\epsilon)  = & f_{j,1,k}^{\epsilon}, \label{EqnP1} \\
Y_0(kr_{j,2}){\rm Re}(u_{j,k}^\epsilon)- J_0(kr_{j,2}){\rm Im}(u_{j,k}^\epsilon)  = & f_{j,2,k}^{\epsilon}, \label{EqnP2}
\end{align}
where
\begin{equation}\label{ratiodelta}
\begin{array}{ll}
  f_{j,\ell,k}^\epsilon
   =\dfrac{2}{c_{j,\ell,k}^\epsilon}\left(|v_{j,\ell,k}^\epsilon|^2-|u_{j,k}^\epsilon|^2\right)-\dfrac{c_{j,\ell,k}^\epsilon}{8}|H^{(1)}_{0}(kr_{j,\ell})|^2,
\\[5mm]\displaystyle
  c_{j,\ell,k}^\epsilon=\frac{\|u^\epsilon(\cdot, k)\|_{j,\infty}}{\|\Phi_k(\cdot,z_{j,\ell})\|_{j,\infty}},
\end{array}
\quad \ell=1,2.
\end{equation}
Here $|u_{j,k}^\epsilon|$ and $|v_{j,\ell,k}^\epsilon|$ ($\ell=1,2$) are measured noisy data satisfying
\begin{equation}\label{Error}
 \begin{array}{ll}
  \||u_{j,k}^\epsilon|-|u_{j,k}|\|_{j,\infty}\leq \epsilon \|u_{j,k}\|_{j,\infty},
  \\[1mm]
  \left\||v_{j,\ell,k}^\epsilon|-|\hat{v}_{j,\ell,k}|\right\|_{j,\infty}\leq \epsilon
     \left\|\hat{v}_{j,\ell,k}\right\|_{j,\infty},
 \end{array}
\end{equation}
where $0<\epsilon<1, \ell=1,2$ and $\hat{v}_{j,\ell,k}(\cdot)=u_{j,k}(\cdot)-c^\epsilon_{j,\ell,k}\Phi_k(\cdot,z_{j,\ell})$.
It is readily to see that the solutions to the perturbed equations \eqref{EqnP1} and \eqref{EqnP2} can also
be given by  formula \eqref{formula} with $f_{j,\ell,k}^\epsilon$ in place of $f_{j,\ell,k}$ ($\ell=1,2$).
Then we have the following stability result.
\begin{theorem}\label{Stablilty}
Under \eqref{assumption},   we have the following estimate
\begin{equation}\label{ErrEstimate}
 \|u_{j,k}^\epsilon-u_{j,k}\|_{j,\infty} \leq \epsilon C_\epsilon\|u_{j,k}\|_{j,\infty},
 \quad j=1,\cdots,m,
\end{equation}
where
\[
C_\epsilon=\frac{2.5(2+\epsilon)^2(3+\epsilon)+15}{1-\epsilon}.
\]
\end{theorem}
\begin{proof}
From \eqref{ratio}, \eqref{ratiodelta}, \eqref{Error} and $v_{j,\ell}(\cdot,k)=u(\cdot,k)-c_{j,\ell,k}\Phi_k(\cdot,z_{j,\ell})$ , we see
\begin{equation}\label{Errratio}
  |c_{j,\ell,k}^\epsilon-c_{j,\ell,k}|\leq \epsilon c_{j,\ell,k},
\end{equation}
and
\begin{align}
  \left\||v_{j,\ell,k}^\epsilon|-|v_{j,\ell,k}|\right\|_{j,\infty}\leq& \left\||v_{j,\ell,k}^\epsilon|-|\hat{v}_{j,\ell,k}|\right\|_{j,\infty}
  +\left\||\hat{v}_{j,\ell,k}^\epsilon|-|v_{j,\ell,k}|\right\|_{j,\infty} \notag \\
  \leq&  \epsilon \left\|\hat{v}_{j,\ell,k}\right\|_{j,\infty}+\epsilon c_{j,\ell,k}\|\Phi_k(\cdot,z_{j,\ell})\|_{j,\infty} \nonumber \\
  \leq&  \epsilon \left(1+ \frac{c_{j,\ell,k}^\epsilon}{c_{j,\ell,k}}\right)\|u_{j,k}\|_{j,\infty}+ \epsilon \|u_{j,k}\|_{j,\infty} \nonumber \\
  \leq&  \epsilon (3+\epsilon)\|u_{j,k}\|_{j,\infty}. \label{Errvjl}
\end{align}
Further, by using \eqref{Error}, \eqref{Errratio}, \eqref{Errvjl} and  $v_{j,\ell}(\cdot,k)=u(\cdot,k)-c_{j,\ell,k}\Phi_k(\cdot,z_{j,\ell})$,
we deduce
\begin{align*}
  \left|\frac{|v_{j,\ell,k}^\epsilon|^2}{c_{j,\ell,k}^\epsilon}-\frac{|v_{j,\ell,k}|^2}{c_{j,\ell,k}}\right|
  \leq& \frac{1}{c_{j,\ell,k}^\epsilon}\left||v_{j,\ell,k}^\epsilon|^2-|v_{j,\ell,k}|^2\right|
   +\left|\frac{1}{c_{j,\ell,k}^\epsilon}-\frac{1}{c_{j,\ell,k}}\right||v_{j,\ell,k}|^2 \\
  \leq & \frac{\epsilon(3+\epsilon)}{c_{j,\ell,k}^\epsilon} \|u_{j,k}\|_{j,\infty}
 \left(|v_{j,\ell,k}^\epsilon|+|v_{j,\ell,k}|\right) \\
 & +\frac{\left|c_{j,\ell,k}^\epsilon-c_{j,\ell,k}\right|}{c_{j,\ell,k}^\epsilon c_{j,\ell,k}}|v_{j,\ell,k}|^2 \\
 \leq &\frac{\epsilon(3+\epsilon)(2+\epsilon)^2}{c_{j,\ell,k}^\epsilon} \|u_{j,k}\|^2_{j,\infty}
   +\frac{\epsilon}{c_{j,\ell,k}^\epsilon}\|v_{j,\ell,k}\|_{j,\infty}^2 \\
 \leq & \frac{\epsilon ((3+\epsilon)(2+\epsilon)^2+2)}{c_{j,\ell,k} (1-\epsilon)}  \|u_{j,k}\|_{j,\infty}^2 \\
 =& \frac{\epsilon((3+\epsilon)(2+\epsilon)^2+2)}{1-\epsilon} \|u_{j,k}\|_{j,\infty}   \|\Phi_k(\cdot,z_{j,\ell})\|_{j,\infty}.
\end{align*}
Similarly, we have
\begin{align*}
  \left|\frac{|u_{j,k}^\epsilon|^2}{c_{j,\ell,k}^\epsilon}-\frac{|u_{j,k}|^2}{c_{j,\ell,k}}\right|
  \leq & \frac{\epsilon(3+\epsilon)}{1-\epsilon} \|u_{j,k}\|_{j,\infty}   \|\Phi_k(\cdot,z_{j,\ell})\|_{j,\infty}, \\
  \frac{\left|c_{j,\ell,k}^\epsilon-c_{j,\ell,k}\right|}{8}\left|H^{(1)}_{0}(kr_{j,\ell})\right|^2
  \leq & 2\epsilon  \|u_{j,k}\|_{j,\infty}   \|\Phi_k(\cdot,z_{j,\ell})\|_{j,\infty}.
\end{align*}
By using the triangle inequality, we obtain
\begin{equation}\label{Thm2est1}
 \left|f_{j,\ell,k}^\epsilon-f_{j,\ell,k}\right|
     \leq  \epsilon \eta_\epsilon   \|u_{j,k}\|_{j,\infty}   \|\Phi_k(\cdot,z_{j,\ell})\|_{j,\infty}, \quad \ell=1,2,
\end{equation}
where
\[
\eta_\epsilon =\frac{2(2+\epsilon)^2(3+\epsilon)+12}{1-\epsilon}.
\]

The subsequent proof is divided into two cases.

Case (i): $k\in\mathbb{K}_N\backslash\{k^*\}$.  

From \eqref{Estimate}, \eqref{rjl}, \eqref{krj1} and \eqref{krj2}, it is readily to see that
\begin{align} 
 \left(\left|J_0(kr_{j,\ell}) \right|^2 + \left|Y_0(kr_{j,\ell}) \right|^2\right)^{1/2} \leq& \sqrt{\frac{2}{\pi kr_{j,\ell}}}+\frac{(kr_{j,\ell})^{-3/2}}{4\sqrt{2\pi}} \nonumber \\ 
   \leq& \sqrt{\frac{2}{\pi kr_{j,\ell}}}\left(1+\frac{1}{16\pi}\right) \nonumber \\ 
   \leq& \frac{0.82}{\sqrt{kR(1-\lambda_{j,\ell})}}, \label{Thm2est2}
\end{align}
which, together with \eqref{Thm2est1} and \eqref{Thm2est2}, implies
\begin{align}
 &\left|J_0(kr_{j,1}) ( f_{j,2,k}^\epsilon-f_{j,2,k} )- J_0(kr_{j,2})   ( f_{j,1,k}^\epsilon-f_{j,1,k} )\right| \nonumber \\
  \leq &0.82\epsilon \eta_\epsilon\|u_{j,k}\|_{j,\infty}\left(
  \frac{\|\Phi_k(\cdot,z_{j,2})\|_{j,\infty}}{\sqrt{kR(1-\lambda_{j,1})}}
 +\frac{\|\Phi_k(\cdot,z_{j,1})\|_{j,\infty}}{\sqrt{kR(1-\lambda_{j,2})}}\right) \nonumber \\
  \leq &0.82^2\epsilon \eta_\epsilon\|u_{j,k}\|_{j,\infty} 
  \frac{1}{2\sqrt{kR(1-\lambda_{j,1})}\sqrt{kR(1-\lambda_{j,2})}} \nonumber \\
  \leq &\frac{0.83\epsilon \eta_\epsilon}{kR}\|u_{j,k}\|_{j,\infty} \label{Thm2est4}
\end{align}
and
\begin{equation} \label{Thm2est5} 
 \left|Y_0(kr_{j,1}) ( f_{j,2,k}^\epsilon-f_{j,2,k})-Y_0(kr_{j,2})(f_{j,1,k}^\epsilon-f_{j,1,k})\right| \leq \frac{0.83\epsilon  \eta_\epsilon\|u_{j,k}\|_{j,\infty}}{kR}.
\end{equation}

Subtracting \eqref{EqnF1} from \eqref{EqnP1} (resp. \eqref{EqnF2} from \eqref{EqnP2}) yields
\begin{align}
 Y_0(kr_{j,1}){\rm Re}(u_{j,k}^\epsilon-u_{j,k}) - J_0(kr_{j,1}){\rm Im}(u_{j,k}^\epsilon-u_{j,k}) =  & f_{j,1,k}^\epsilon-f_{j,1,k}, \label{EqnErr1} \\
 Y_0(kr_{j,2}){\rm Re}(u_{j,k}^\epsilon-u_{j,k}) - J_0(kr_{j,2}){\rm Im}(u_{j,k}^\epsilon-u_{j,k}) = & f_{j,2,k}^\epsilon-f_{j,2,k}. \label{EqnErr2}
\end{align}
Thus,
\begin{eqnarray*}
{\rm Re}(u_{j,k}^\epsilon-u_{j,k}) & = & \frac{|J_0(kr_{j,1}) (f_{j,2,k}^\epsilon-f_{j,2,k})- J_0(kr_{j,2})(f_{j,1,k}^\epsilon-f_{j,1,k})|}{\det |A_{j,k}|},\\
{\rm Im}(u_{j,k}^\epsilon-u_{j,k}) & = & \frac{|Y_0(kr_{j,1}) ( f_{j,2,k}^\epsilon-f_{j,2,k})-Y_0(kr_{j,2})(f_{j,1,k}^\epsilon-f_{j,1,k})|}{\det |A_{j,k}|}.
\end{eqnarray*}

Then, in terms of \eqref{detAj}, \eqref{Thm2est4} and \eqref{Thm2est5}, we derive
\[
  \|u_{j,k}^\epsilon-u_{j,k}\|_{j, \infty}
   \leq \epsilon C_\epsilon\|u_{j,k}\|_{j,\infty},
\]
where
\[
C_\epsilon=\frac{2.5(2+\epsilon)^2(3+\epsilon)+15}{1-\epsilon}.
\]

Case (ii): $k=k^*$. 

From \eqref{alpha1}, \eqref{alpha2}, \eqref{detEstimate2},   $0.314 \leq t_1\leq 0.346$ and $1.551 \leq t_2\leq 1.571$, we have
\begin{equation} \label{Thm2est6}
   \left|J_0(t_1) \right|  \leq 1- \frac{t_1^2}{4}+ \frac{t_1^4}{64}< 0.976, \quad \left|J_0(t_2) \right|  \leq 1-\frac{t_2^2}{4}+\frac{t_2^4}{64}< 0.49,
\end{equation}
and
\begin{align}
   \left|Y_0(t_\ell) \right|  \leq& \left|\frac{2}{\pi} \left[1-\frac{t_\ell^2}{4}+\frac{t_\ell^4}{64}\right]\left[ \ln \frac{t_\ell}{2} + C_0\right]
   + \frac{t_\ell^2}{2\pi} -\frac{ 3t_\ell^4}{64\pi}\right|
                     +\frac{ t_\ell^6\ln t_\ell}{3619}+\frac{ t_\ell^6}{1809} \nonumber \\ 
  <& 
\begin{cases}
   0.777,  & \ell=1, \\
   0.413,  & \ell=2.
\end{cases} \label{Thm2est7}
\end{align}
 
Further, from \eqref{Thm2est1}, \eqref{Thm2est6} and \eqref{Thm2est7}, we obtain
\begin{align}
 &\left|J_0(t_1) (f_{j,2,k}^\epsilon-f_{j,2,k}) - J_0(t_2) ( f_{j,1,k}^\epsilon-f_{j,1,k})\right| \\[2mm]\nonumber
  \leq & 0.976\epsilon \eta_\epsilon \|u_{j,k}\|_{j,\infty}   \|\Phi_k(\cdot,z_{j,2})\|_{j,\infty}
   +0.49\epsilon \eta_\epsilon  \|u_{j,k}\|_{j,\infty}   \|\Phi_k(\cdot,z_{j,1})\|_{j,\infty}\nonumber \\
  \leq & 0.34 \epsilon \eta_\epsilon  \|u_{j,k}\|_{j,\infty} \label{Thm2est8}
\end{align}
and
\begin{align}
 &\left|Y_0(t_1) ( f_{j,2,k}^\epsilon-f_{j,2,k} )- Y_0(t_2)  (f_{j,1,k}^\epsilon-f_{j,1,k})\right| \nonumber\\
  \leq & 0.777\epsilon C_\epsilon \|u_{j,k}\|_{j,\infty}   \|\Phi_k(\cdot,z_{j,2})\|_{j,\infty}
   +0.413\epsilon C_\epsilon \|u_{j,k}\|_{j,\infty}   \|\Phi_k(\cdot,z_{j,1})\|_{j,\infty} \nonumber\\
  \leq & 0.28\epsilon \eta_\epsilon  \|u_{j,k}\|_{j,\infty}. \label{Thm2est9}
\end{align}
Then, by solving the equations \eqref{EqnErr1} and \eqref{EqnErr2} with \eqref{detAj}, \eqref{Thm2est8} and \eqref{Thm2est9}, we have completed the proof.
\end{proof}

\section{Fourier method}\label{sec:fourier_method}

Once the phase information of the measured data is retrieved, we could turn to the standard inverse source problem:

\begin{problem}[Multifrequency inverse source problem]\label{prob:phased_ISP}
Given the admissible set of wavenumbers $\mathbb{K}_N$, find an approximation of $S$ from the multifrequency data $\{u(x, k): x\in\Gamma_R ,\ k\in\mathbb{K}_N\}$.
\end{problem}

In this section, we will briefly outline the Fourier method for solving Problem \ref{prob:phased_ISP}. For more details on the Fourier method, we refer to \cite{ZG15} for the near-field case and \cite{WGZL17} for the far-field case. The basic idea of the Fourier method is to approximate the source $S$ by a Fourier expansion of the form
\begin{equation}\label{eq:fourier_expansion}
S_N(x):=\sum\limits_{|\bm{\ell}|_{\infty}\leq N}\hat{s}_{\bm\ell}\phi_{\bm\ell}(x)
\end{equation}
where ${\bm\ell}\in\mathbb{Z}^2$, $\hat{s}_{\bm\ell}$ are the Fourier coefficients and
\[
\phi_{\bm\ell}(x):=\exp\left(\mathrm{i} \frac{\pi}{a}{\bm\ell}\cdot x\right),\quad {\bm\ell}\in\mathbb{Z}^2,
\]
are the Fourier basis functions.

Let $\nu_\rho$ be the unit outward normal to $\Gamma_{\rho}:=\{x \in \mathbb{R}^2: |x|=\rho>R\}$ and define
\begin{align*}
w(x, k):= & \sum_{n\in\mathbb{Z}}\frac{H^{(1)}_n(k\rho)}{H^{(1)}_n(kR)}\hat{u}_{k,n}\mathrm{e}^{\mathrm{i} n \theta}, \quad x \in \Gamma_\rho, \\
\partial_{\nu_\rho}w(x, k):= &\sum_{n\in\mathbb{Z}}k\frac{H^{(1)\prime}_n(k\rho)}{H^{(1)}_n(kR)}\hat{u}_{k,n}\mathrm{e}^{\mathrm{i}  n \theta}, \quad x \in \Gamma_\rho.
\end{align*}
where
\[
\hat{u}_{k,n}:=\frac{1}{2\pi}\int_0^{2\pi}u(R,\theta, k)\mathrm{e}^{-\mathrm{i} n \theta}\, \mathrm{d} \theta.
\]
and $u(R,\theta, k)$ stands for $u(x, k)|_{\Gamma_R}$ in polar coordinates $x=R(\cos\theta,\sin\theta)$.

Following \cite{ZG15}, the Fourier coefficients are given by
\begin{align*}
 \hat{s}_{\bm\ell}= & \displaystyle\frac{1}{4a^2}\int_{\Gamma_\rho}\left(\partial_{\nu_\rho}
w(x, k)+ \mathrm{i} \frac{\pi}{a}(\bm{\ell}\cdot\nu_\rho)w(x, k)\right)\overline{\phi_{\bm\ell}(x)}\, \mathrm{d} s_x,\quad 1\leq |\bm{\ell}|_{\infty}\leq N, \\
 \hat{s}_{\bm 0}= &\displaystyle\frac{\lambda\pi}{4a^2\sin \lambda\pi}\Bigg(\int_{\Gamma_\rho}\left(\partial_{\nu_\rho}w(x, k^*)+ \mathrm{i} \frac{\pi}{a}({\bm\ell}^*\cdot\nu_\rho)w(x, k^*)\right)
 \overline{\phi_{{\bm\ell}^*}(x)}\, \mathrm{d} s_x\\
 &-\sum\limits_{1\leq|\bm{\ell}|_\infty\leq N} \hat{s}_{\bm\ell}\int_{V_0}\phi_{\bm\ell}(x)
 \overline{\phi_{\bm{\ell}^*}(x)}\, \mathrm{d} x\Bigg).
\end{align*}
where $\lambda$ is a constant satisfying $0<\lambda<a/(2\pi)$, $\bm{\ell}^*=(\lambda,0)$ and $k^*=\pi\lambda/a$.

The Fourier method is summarized in {\bf Algorithm FM}.
\begin{table}[h]
\centering
\begin{tabular}{cp{.8\textwidth}}
\toprule
\multicolumn{2}{l}{{\bf Algorithm FM:}\quad Fourier method for recovering the source} \\
\midrule
 {\bf Step 1} & Choose the parameters $\lambda, \rho, N$ and the admissible set $\mathbb{K}_N$;  \\
{\bf Step 2} &  Collect the noisy multifrequency data $\{u^\epsilon(x, k): x\in\Gamma_R,\ k\in\mathbb{K}_N\}$; \\
{\bf Step 3} & Evaluate the auxiliary Cauchy data $w$ and $\partial_{\nu_\rho}w$ on $\Gamma_\rho$;\\
{\bf Step 4} & Compute the Fourier coefficients $\hat{s}_{\bm 0}$ and $\hat{s}_{\bm\ell}(1\leq |\bm{\ell}|_{\infty}\leq N)$ and then $S_N$ defined in \eqref{eq:fourier_expansion} is the reconstruction of $S$. \\
\bottomrule
\end{tabular}
\end{table}

For the theoretical justifications of the Fourier method as well as validating numerical examples, we refer to \cite{ZG15}.

\section{Numerical examples}\label{sec:numerics}

In this section, we present below several numerical results to illustrate the feasibility and effectiveness of the proposed method. 

Synthetic radiated fields were generated by solving the forward problem via the finite element method. The computation was implemented on a truncated circular domain enclosed by an absorbing perfectly matched layer. The mesh of the forward solver is successively refined till the relative error of the successive measured data is below $0.1\%$. In addition to the discretization error of the forward solver, we also add some uniformly distributed random noises to the measurements in order to test the stability of the algorithm. The noisy phaseless data was given by the following formula:
\[
u^\epsilon:=(1+\epsilon r)|u|
\]
where $r$ is a uniformly distributed random number ranging from $-1$ to $1$, and $\epsilon>0$ is the noise level.

\subsection{Validation of phase retrieval}\label{retrieval_result}

In this subsection, we will present a numerical validation of the phase retrieval technique. In our experiment, the test source function is chosen as the following mountain shaped source
\begin{align*}
S(x_1, x_2)=&1.1\exp(-200((x_1-0.01)^2+(x_2-0.12)^2)) \\
&-100(x_2^2-x_1^2)\exp(-90(x_1^2+x_2^2)).
\end{align*}

Figure \ref{fig: source_1} illustrates the surface and contour plots of the exact source function $S$ in the domain $V_0=[-0.3, 0.3]^2$.

\begin{figure}
   \centering
   \subfloat[]{\includegraphics[width=.45\textwidth]{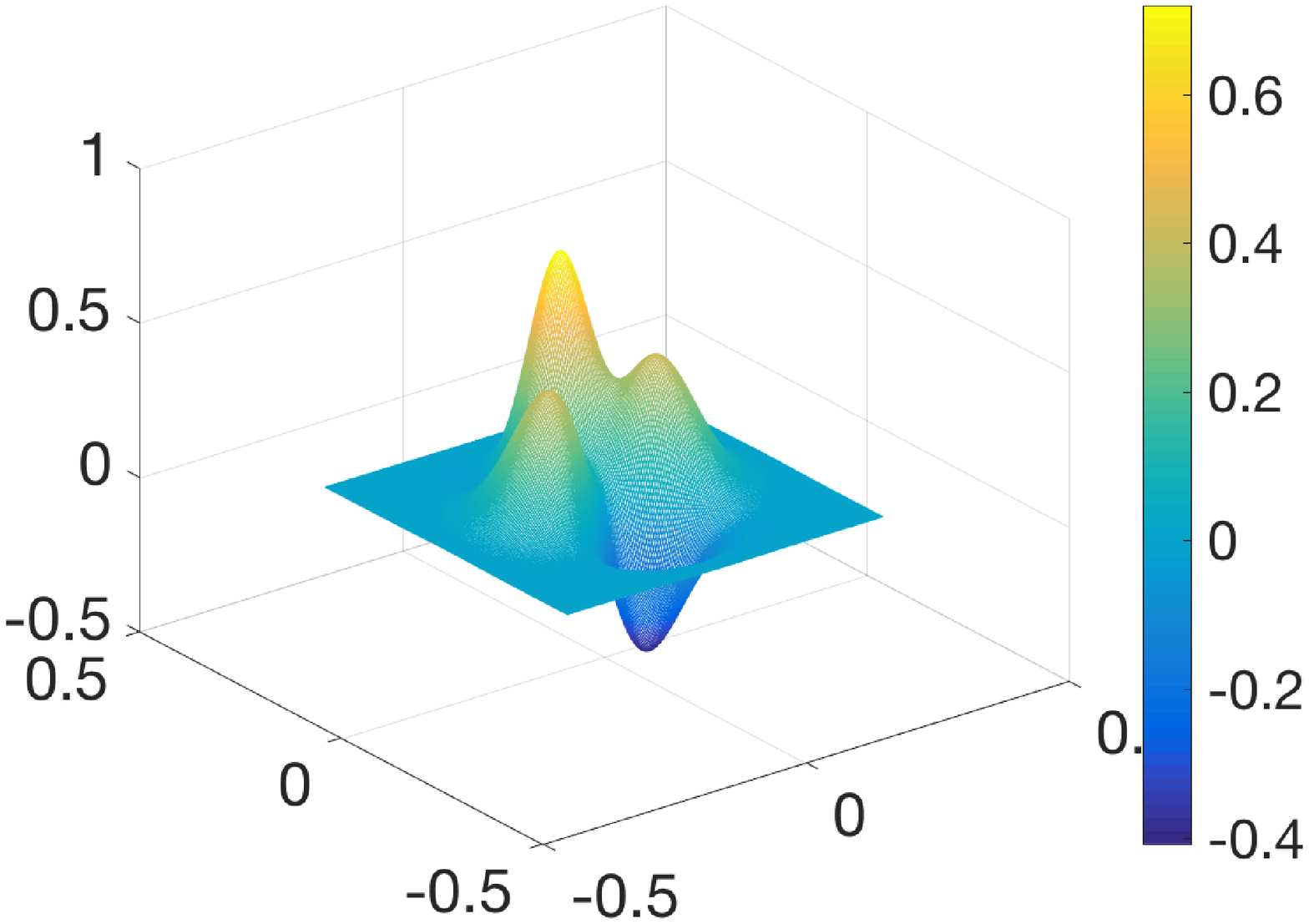}}
   \subfloat[]{\includegraphics[width=.45\textwidth]{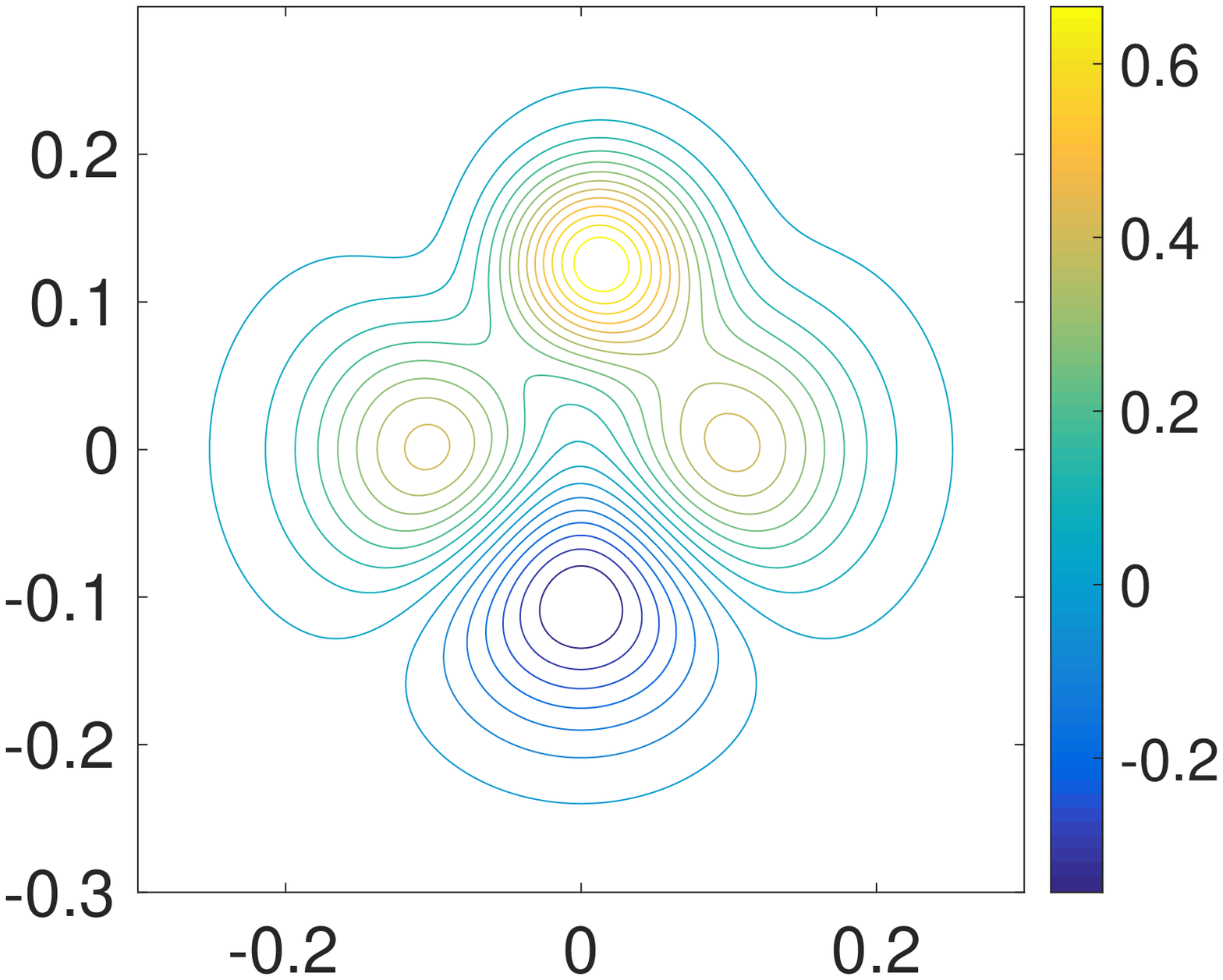}}
   \caption{The exact source function $S$.  (a) surface plot (b) contour plot}
   \label{fig: source_1}
\end{figure}

First we specify the parameters used in \eqref{assumption}: $a=0.3, \tau= 6, R=\tau a=1.8, m=10, \lambda_{1,1}=0.5, \vartheta_1=0.1\pi$ and thus $k^*=\pi/9$. The other wavenumbers are taken as $k=10\pi/3, 50\pi/3$ and $100\pi/3$. Without loss of generality, here we only show the phase retrieval results on $\Gamma_1$. Figure \ref{fig: retrieval_setup} shows the geometry setup of this example for two typical wavenumbers.

\begin{figure}
   \centering
   \subfloat[]{\includegraphics[width=.45\textwidth]{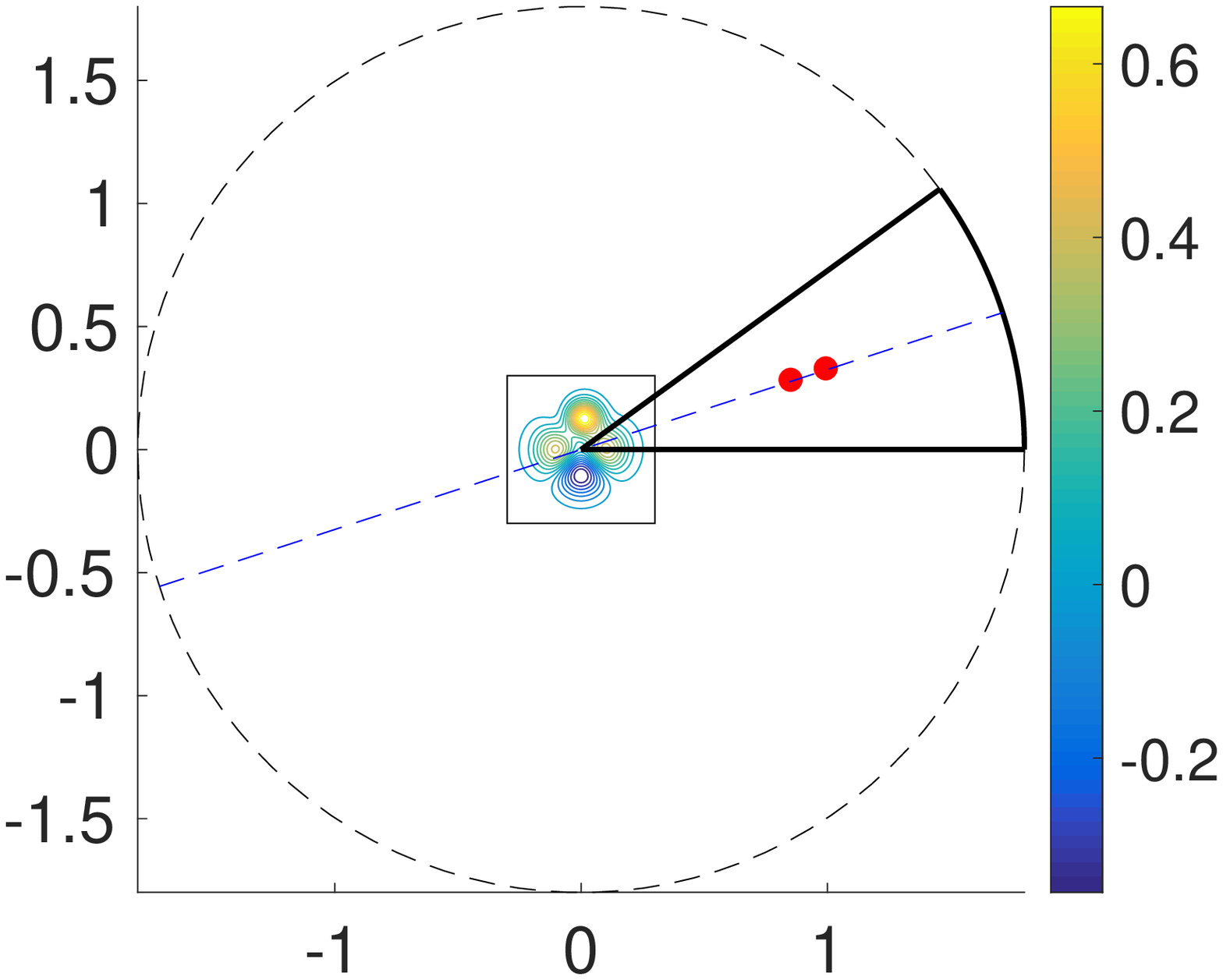}}
   \subfloat[]{\includegraphics[width=.45\textwidth]{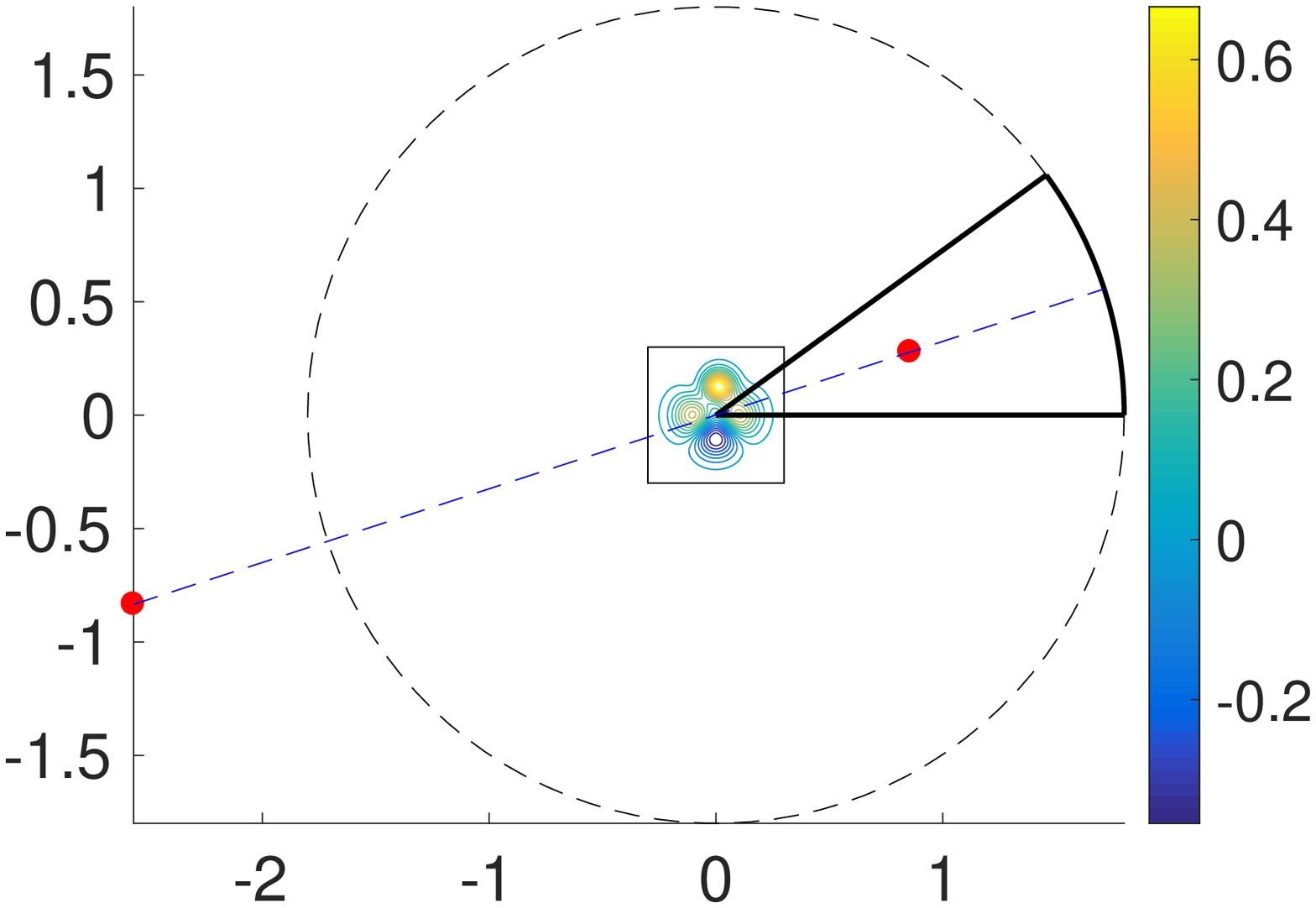}}
   \caption{Geometry illustration of the phase retrieval on $\Gamma_1$. The boundary of sector $B_1$ is depicted by the thick black line. The reference source points are denoted by the small red points. (a) $k=10\pi/3$ (b) $k=k^*=\pi/9$.}
   \label{fig: retrieval_setup}
\end{figure}

To quantitatively evaluate the accuracy of the phase retrieval, we list the relative errors in Table \ref{tab:error_L2} and \ref{tab:error_infinity}. In Table \ref{tab:error_L2} and Table \ref{tab:error_infinity}, the discrete relative errors are respectively calculated as follows
\begin{eqnarray*}
\frac{\left(\sum\limits_{n=1}^{N_j} |u(x_n, k)-u^\epsilon(x_n, k)|^2\right)^{1/2}}{\left(\sum\limits_{n=1}^{N_j} |u(x_n, k)|^2\right)^{1/2}},\quad x_n\in \Gamma_j,\ j=1,\cdots,m, \\
\frac{\max\limits_{x_n\in \Gamma_j} |u(x_n, k)-u^\epsilon(x_n, k)|}{\max\limits_{x_n\in \Gamma_j} |u(x_n, k)|},\quad n=1,\cdots,N_j,\ j=1,\cdots,m. 
\end{eqnarray*}
where $N_j$ denotes the number of measurements on $\Gamma_j$,  and $u$ and $u_\epsilon$ are the exact and retrieved radiated fields, respectively. The results in  Table \ref{tab:error_L2} and Table \ref{tab:error_infinity} could be served to verify our theoretical findings that the error decreases as the noise level decreases. Although the accuracy deteriorates correspondingly as the noise level increases, it can be seen from these results that the phase retrieval procedure is quite accurate and stable as long as the perturbation is moderately small. 

\begin{table}
\centering
\begin{tabular}{ccccc}
        \bottomrule
         & $k^*=\dfrac{\pi}{9}$ &  $k=\dfrac{10\pi}{3}$ & $k=\displaystyle\frac{50\pi}{3}$ &  $k=\dfrac{100\pi}{3}$\\
        \midrule
    $\epsilon=0$ & $5.2\times 10^{-16}$ & $1.9\times 10^{-16}$ & $3.6\times 10^{-16}$ &  $5.0\times 10^{-16}$ \\
    $\epsilon=0.1\%$  & 0.28\% & 0.13\% & 0.22\% &  0.27\% \\
    $\epsilon=1\%  $  & 2.95\% & 1.22\% & 2.16\% &  2.08\% \\
    $\epsilon=5\%  $  & 15.5\% & 5.12\% & 10.6\% &  10.23\% \\
        \bottomrule
\end{tabular}
\caption{\label{tab:error_L2} The relative $L^2$ errors between $u(\cdot, k)|_{\Gamma_1}$ and $u^\epsilon(\cdot, k)|_{\Gamma_1}$ for different wavenumbers and noise levels.}
\end{table}

\begin{table}
\centering
\begin{tabular}{ccccc}
        \bottomrule
         & $k^*=\dfrac{\pi}{9}$ &  $k=\dfrac{10\pi}{3}$ & $k=\dfrac{50\pi}{3}$ &  $k=\dfrac{100\pi}{3}$\\
        \midrule
    $\epsilon=0$ & $1.1\times 10^{-15}$ & $3.1\times 10^{-16}$ & $5.5\times 10^{-16}$ &  $3.8\times 10^{-16}$ \\
    $\epsilon=0.1\%$  & 0.45\% & 0.16\% & 0.44\% &  0.24\% \\
    $\epsilon=1\%  $  & 4.60\% & 1.42\% & 3.32\% &  3.18\% \\
    $\epsilon=5\%  $  & 23.9\% & 6.47\% & 19.2\% &  18.97\% \\
        \bottomrule
    \end{tabular}
\caption{\label{tab:error_infinity} The relative infinity errors between $u(\cdot, k)|_{\Gamma_1}$ and $u^\epsilon(\cdot, k)|_{\Gamma_1}$ for different wavenumbers and noise levels.}
\end{table}

\subsection{Reconstruct the source}

In this subsection, we are going to show some numerical results of the Fourier method for imaging the source function with phased data. The phased radiated field is obtained by the phase retrieval technique in the previous subsection. Following \cite{ZG15}, the truncation of the Fourier expansion is chosen as $N=2\lceil\epsilon^{-1/3}\rceil$ where the ceiling function $\lceil X\rceil$ denotes the largest integer that is smaller than $X+1$. We use 400 pseudo-uniformly distributed measurements on $\Gamma_R$ and $\Gamma_\rho$ with $\rho=1.4$. The relative $L^2$ error of the reconstruction is evaluated over $V_0$ with $800\times 800$ equally spaced grid. For more implementation details, we refer to \cite{ZG15}.

\begin{figure}
   \centering
   \subfloat[]{\includegraphics[width=.45\textwidth]{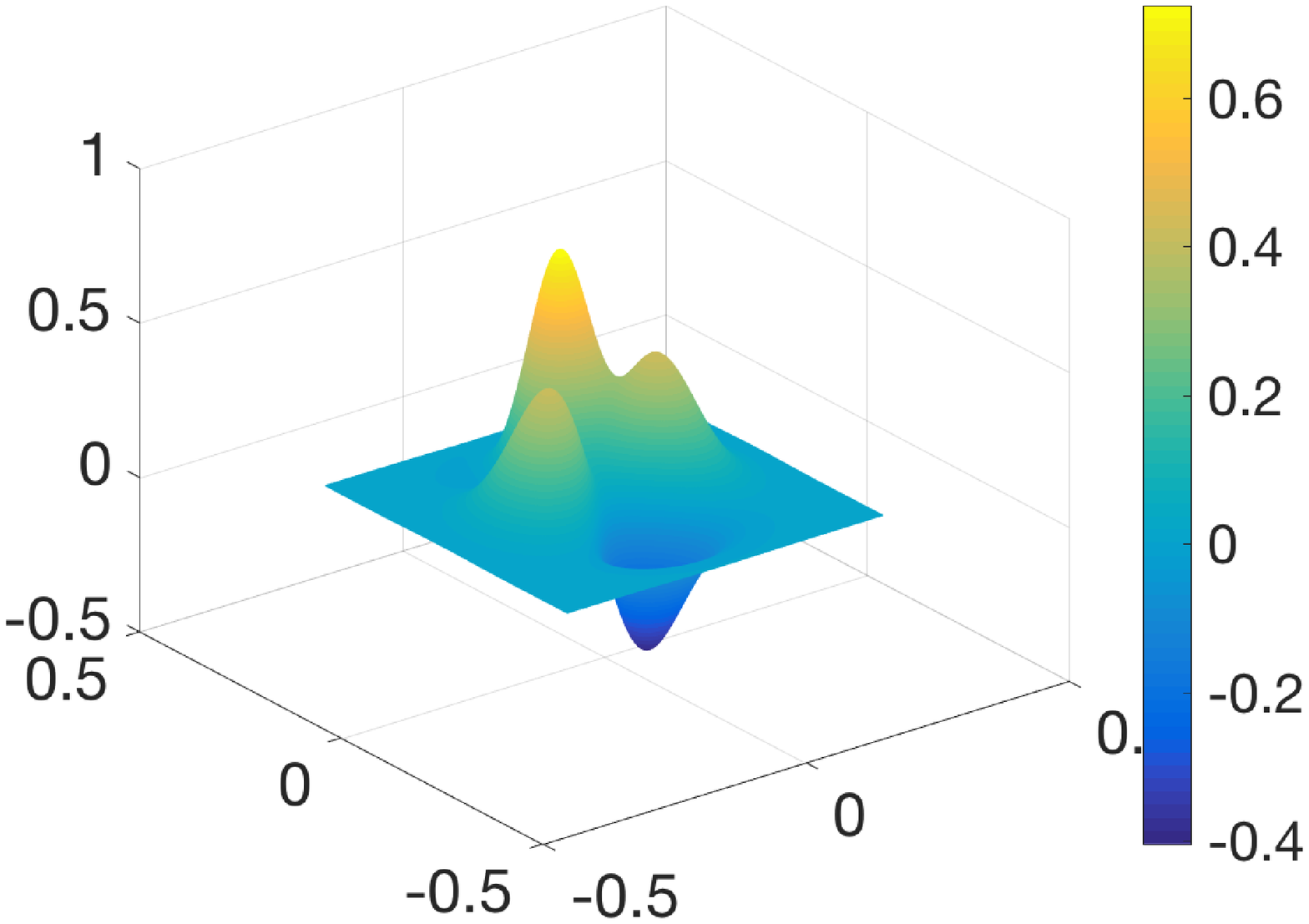}}
   \subfloat[]{\includegraphics[width=.45\textwidth]{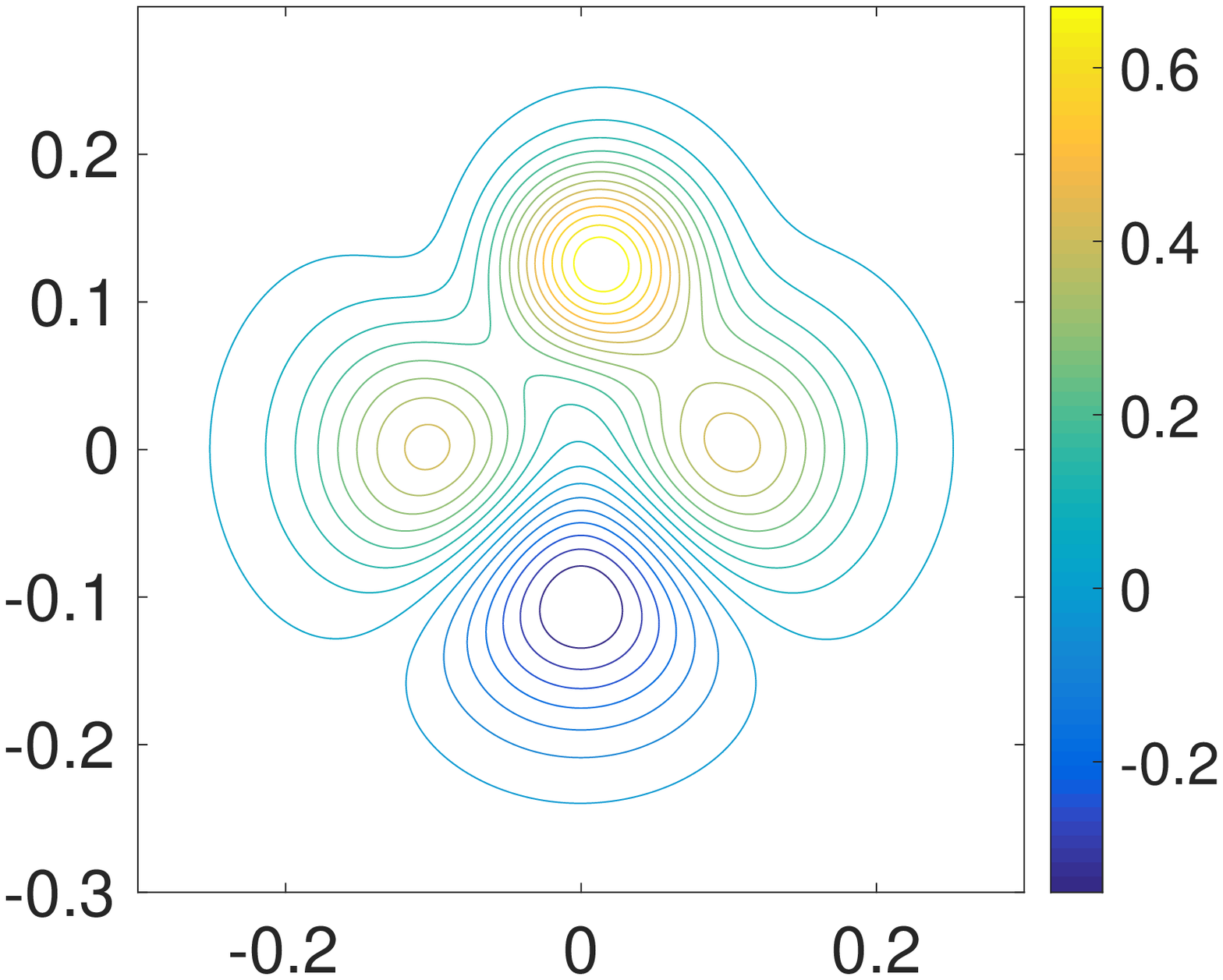}}
   \caption{The reconstruction of source $S$ with $1\%$ noise.  (a) surface plot (b) contour plot}
   \label{fig:reconstruction}
\end{figure}

The final reconstruction of the source $S$ with $1\%$ noise is shown in Figure \ref{fig:reconstruction}. Our experiment shows that, using the phaseless measurements polluted with $1\%$, $2\%$ and $5\%$ random noise, the relative $L^2$ error of the final source reconstruction is $4.60\%$, $5.82\%$ and $8.39\%$, respectively. As seen from subsection \ref{retrieval_result},  for moderately large noise level, say $5\%$, the retrieved radiated data is not very accurate. Nevertheless, the final reconstruction of the source function is not drastically affected and thus satisfactory. This salient capability is due to the fact that the Fourier coefficients are evaluated by computing integrals which could significantly filter the effect of noise. Therefore, the Fourier method is inherently tolerant to large noise.

\section{Conclusion}\label{sec:conclusion}

In this work,  we have proposed a numerical method for determining acoustic sources from multifrequency phaseless measurements. Our reconstruction procedure consists of two stages and it does not rely on any regularization or iteration. On the first stage the phase information of the measured data is retrieved via the newly introduced reference source technique. As a result, we only need to solve the standard phased inverse source problem on the second stage.

Since the Fourier method is applicable to the scalar 3D Helmholtz equation as well as the vector Maxwell's system \cite{WGZL17, WMGL17, WSGLL18}, our future work consists of extending the idea of phase retrieval to the 3D case and the vector electromagnetic model.

\section*{Acknowledgements}

The work of D. Zhang was supported by the NSFC grants under No. 11671170. The work of Y. Guo was supported by the NSFC grants under Nos. 11671111, 41474102 and 11601107. The work of J. Li was supported by the NSFC grant under 11571161 and the Shenzhen Sci-Tech Fund under JCYJ20160530184212170. The work of H. Liu was supported by the startup fund and FRG grants from Hong Kong Baptist University, Hong Kong RGC General Research Funds, No. 12302415 and 12302017. 



\begin{thebibliography}{10}

\bibitem{AM06} R. Albanese and P. Monk, The inverse source problem for Maxwell's equations, {\it  Inverse Problems}, {\bf 22}, 1023--1035, 2006.

\bibitem{AHLS17} A. Alzaalig, G. Hu, X. Liu and J. Sun, Fast acoustic source imaging using multi-frequency sparse data, {\it arXiv:1712.02654v1}, 2017.

\bibitem{ABF02} H. Ammari, G. Bao and J. Fleming, An inverse source problem for Maxwell's equations in magnetoencephalography,  {\it SIAM J. Appl. Math.}, {\bf 62}, 1369--1382, 2002.

\bibitem{AZML07} M. A. Anastasio, J. Zhang, D. Modgil and P. J. La Rivi\`{e}re, Application of inverse source concepts to photoacoustic tomography, {\it Inverse Problems}, {\bf 23}, S21--S35, 2007.

\bibitem{AKK91} T. Angel, A. Kirsch and R. Kleinmann, Antenna control and generalized characteristic modes, {\it Proc. IEEE}, {\bf 79}, 1559--1568, 1991.

\bibitem{Arr99} S. R. Arridge, Optical tomography in medical imaging, {\it Inverse Problems}, {\bf 15}, R41--R93, 1999.

\bibitem{BLT11} G. Bao, J. Lin and F. Triki, Numerical solution of the inverse source problem for the Helmholtz equation with multiple frequency data, {\it Contemp. Math. AMS}, {\bf 548}, 45--60, 2011.

\bibitem{BLLT15} G. Bao, P. Li, J. Lin and F. Triki, Inverse scattering problems with multi-frequencies, {\it Inverse Problems}, {\bf 31}, 093001, 2015.

\bibitem{BLRX15} G. Bao, S. Lu, W. Rundell and B. Xu, A recursive algorithm for multifrequency acoustic inverse source problems, {\it SIAM J. Numer. Anal.}, {\bf 53}, 1608--1628, 2015.

\bibitem{BZ16} G. Bao and L. Zhang, Shape reconstruction of the multi-scale rough surface from multi-frequency phaseless data, {\it Inverse Problems}, {\bf 32}, 085002, 2016.

\bibitem{BC77} N. Bleistein and J. Cohen, Nonuniqueness in the inverse source
problem in acoustics and electromagnetics, {\it J. Math. Phys.}, {\bf 18}, 194--201, 1977.



\bibitem{DLU17} Y. Deng, H. Liu and G. Uhlmann, On an inverse boundary problem arising in brain imaging, {\it arXiv:1702.00154v3}, 2017.

\bibitem{DML07} A. Devaney, E. Marengo and M. Li, The inverse source problem in nonhomogeneous background media, {\it SIAM J. Appl. Math.}, {\bf 67}, 1353--1378, 2007.

\bibitem{El1} A. E. Badia and T. Ha-Duong, An inverse source problem in potential
analysis, {\it Inverse Problems}, {\bf 16}, 651--663, 2000.

\bibitem{El2} A. E. Badia and T. Nara, An inverse source problem for Helmholtz's
equation from the Cauchy data with a single wave number, {\it Inverse Problems}, {\bf 27}, 105001, 2011.

\bibitem{EV09} M. Eller and N. Valdivia, Acoustic source identification using
multiple frequency information, {\it Inverse Problems}, {\bf 25}, 115005, 2009.


\bibitem{GHR12} R. Griesmaier, M. Hanke and T. Raasch, Inverse source problems for the Helmholtz equation and the windowed Fourier transform, {\it SIAM J. Sci. Comput.}, {\bf 34}(3), A1544--A1562, 2012.

\bibitem{GHR13} R. Griesmaier, M. Hanke and T. Raasch, Inverse source problems for the Helmholtz equation and the windowed Fourier transform II, {\it SIAM J. Sci. Comput.}, {\bf 35}(5), A2188--A2206, 2013.

\bibitem{HR98} S. He and V. Romanov, Identification of dipole sources in a bounded domain for Maxwell's equations, {\it Wave Motion}, {\bf 28}, 25--44, 1998.

\bibitem{IK11} O. Ivanyshyn and R. Kress, Inverse scattering for surface impedance from phaseless far field data, {\it J. Comput. Phys}, {\bf 230}, 3443--3452, 2011.

\bibitem{KR17} M. V. Klibanov and V. G. Romanov,  Uniqueness of a 3-D coefficient inverse scattering problem without the phase information, {\em Inverse Problems}, {\bf 33}, 095007, 2017.

\bibitem{LLZ09} J. Li, H. Liu and J. Zou, Strengthened linear sampling
method with a reference ball, {\em SIAM J. Sci. Comp.}, {\bf 31}, 4013--4040, 2009.

\bibitem{LU15} H. Liu and G. Uhlmann,  Determining both sound speed and internal source in thermo- and photo-acoustic tomography, {\it Inverse Problems}, {\bf 31}, 105005, 2015.

\bibitem{Olv74} F. W. J. Olver, {\it Asymptotics and Special Functions},  (New York and London: Academic Press), 1974.


\bibitem{WMGL17} G. Wang, F. Ma, Y. Guo and J. Li, Solving the multi-frequency electromagnetic inverse source problem by the Fourier method, {\it J. Differential Equations}, DOI: 10.1016/j.jde.2018.02.036

\bibitem{WGZL17} X. Wang, Y. Guo, D. Zhang and H. Liu, Fourier method for recovering acoustic sources from multi-frequency far-field data, {\it Inverse Problems}, {\bf  33}, 035001, 2017.

\bibitem{WSGLL18} X. Wang, M. Song, Y. Guo, H. Li and H. Liu, Fourier method for identifying electromagnetic sources with multi-frequency far-field data, {\it arXiv:1801.03263v1}, 2018.

\bibitem{Watson} G. N. Watson, {\it A Treatise on The Theory of Bessel Functions}, (Cambridge: Cambridge University), 1952.

\bibitem{ZZ17} B. Zhang and H. Zhang, Recovering scattering obstacles by multi-frequency phaseless far-field data, {\it Journal of Computational Physics}, {\bf 345}, 58--73, 2017.

\bibitem{ZG15} D. Zhang and Y. Guo, Fourier method for solving the multi-frequency inverse source problem for the Helmholtz equation, {\it Inverse Problems}, {\bf  31}, 035007, 2015.

\bibitem{ZGLL17} D. Zhang, Y. Guo, J. Li and H. Liu, Locating multiple multipolar acoustic sources using the direct sampling method, {\it Communications in Computational Physics}, Accepted. 

\end{thebibliography}
\end{document}